\documentclass[12pt,reqno]{amsart}

\usepackage{a4wide}
\usepackage{amsmath,amssymb} 
\usepackage{bbm}
\usepackage{color}
\usepackage{tikz}
\usepackage{tkz-graph}
\usepackage{tikz-cd} 
\usepackage{dsfont}
\usepackage{euscript,graphics,amssymb,comment}
\usepackage{pdfpages}
\usepackage{graphicx}
\usepackage{fancyhdr}

\usepackage{bm}

\theoremstyle{plain}
\newtheorem{theorem}{Theorem}[section]
\newtheorem{prop}[theorem]{Proposition}
\newtheorem{lemma}[theorem]{Lemma}
\newtheorem{coro}[theorem]{Corollary}

\theoremstyle{definition}

\newtheorem{remark}[theorem]{Remark}

\newcommand{\p}{\mathbf{p}}

\newcommand{\R}{{\mathbb R}}
\newcommand{\N}{{\mathbb N}}

\newcommand{\mc}{\mathcal}

\newcommand{\dd}{\,\mathrm{d}}

\newcommand{\ts}{\hspace{0.5pt}}
\newcommand{\TT}{\mathbb{T}}

\usepackage{mathtools}

\DeclareMathSymbol{\varnothing}{\mathord}{AMSb}{"3F}

\begin{document}

\title{Fast dimension spectrum for a potential with a logarithmic singularity}

\author{Philipp Gohlke}
\address{Lund University, Centre for Mathematical Sciences, \newline
\hspace*{\parindent}Box 118, 221 00 Lund, Sweden}
\email{philipp\_nicolai.gohlke@math.lth.se}

\author{Georgios Lamprinakis}
\address{Lund University, Centre for Mathematical Sciences, \newline
\hspace*{\parindent}Box 118, 221 00 Lund, Sweden}
\email{georgios.lamprinakis@math.lth.se}

\author{J\"{o}rg Schmeling}
\address{Lund University, Centre for Mathematical Sciences, \newline
\hspace*{\parindent}Box 118, 221 00 Lund, Sweden}
\email{joerg@math.lth.se}

\begin{abstract}
We regard the classic Thue--Morse diffraction measure as an equilibrium measure for a potential function with a logarithmic singularity over the doubling map. Our focus is on unusually fast scaling of the Birkhoff sums (superlinear) and of the local measure decay (superpolynomial). 
For several scaling functions, we show that points with this behavior are abundant in the sense of full Hausdorff dimension. 
At the fastest possible scaling, the corresponding rates reveal several remarkable phenomena. There is a gap between level sets for dyadic rationals and non-dyadic points, and beyond dyadic rationals, non-zero accumulation points occur only within intervals of positive length. The dependence between the smallest and the largest accumulation point also manifests itself in a non-trivial joint dimension spectrum.
\end{abstract}



\maketitle

\section{Introduction and main results}

The study of potential functions $\psi$ over an expanding dynamical system $(X,T)$ and the corresponding equilibrium measures has a long and rich history; for a few classical references relevant for this work compare \cite{Bowen,PW,Ruelle}. 
If the potential function $\psi$ is sufficiently regular, the full strength of the thermodynamic formalism is applicable. Using standard results in multifractal analysis, this yields a detailed description of both the Birkhoff averages of the potential function and of the local dimensions of the equilibrium measure. More precisely, one considers
\[
b_{\psi}(x) = \lim_{n \to \infty} \frac{1}{n} S_n \psi(x), \quad S_n \psi(x) = \sum_{m=0}^{n-1} \psi(T^m x),
\]
and the corresponding dimension spectrum, which is given by the Hausdorff dimension of the corresponding level sets,
\[
f_{\psi}(\beta) = \dim_H \{x \in X : b_{\psi}(x) = \beta \}.
\]
If $\psi$ is H\"{o}lder continuous (and the dynamical system sufficiently nice), the dimension spectrum $f_{\psi}$ is known to be given by a concave real analytic function, supported on a finite interval, outside of which the level sets are empty \cite{PW}. In this setting, the local dimension
\[
d_{\mu}(x) = \lim_{r\to 0} \frac{\log \mu(B_r(x))}{\log(r)}
\]
 of the unique equilibrium measure $\mu$ coincides with the Birkhoff average $b_{\psi}(x)$ up to a constant (whenever any of the limits exists). A multifractal analysis of $d_{\mu}$ is therefore obtained along the same lines.\\
Over the last decades, similar results have been established under less restrictive regularity assumptions. 
At the same time, the study of singular (or unbounded) potentials has gained increased attention. In the presence of a singularity, the dimension spectra can be positive on a half-line and the points with infinite Birkhoff averages (or infinite local dimensions of the equilibrium measure) may have full Hausdorff dimension. In this case, a more complete understanding can be obtained by renormalizing the Birkhoff sums (or the measure decay on shrinking balls) with a more quickly increasing function. This was studied for the specific case of the Saint-Petersburg potential in \cite{KLRW} and in the context of continued fraction expansions; see for example \cite{FLWW,LR}.

In this note, we contribute to the study of singular potentials and their equilibrium measures via a case study of the Thue--Morse (TM) measure. This measure  was one of the first examples of a singular continuous measure, exhibited by Mahler almost a century ago \cite{Mahler}. To this day, it is of interest in number theory and the study of substitution dynamical systems and continues to be the object of active research---compare the review \cite{Queff} for a collection of recent results and open questions. It can be written as an infinite Riesz product on the torus $\TT$ (identified with the unit interval) via
\[
\mu_{\operatorname{TM}} = \prod_{m=0}^{\infty} \bigl( 1 - \cos(2\pi 2^m x) \bigr),
\]
to be understood as a weak limit of absolutely continuous probability measures. 
The TM-measure falls into the class of $g$-measures \cite{Keane}, most recently renamed ``Doeblin measures" in \cite{BCJOE}, giving credit to the pioneering role of Doeblin and Fortet \cite{DF}. This class of measures had an important role in fueling the development of the thermodynamic formalism, largely due to the contributions by Walters \cite{WaltersI,WaltersII} and Ledrappier \cite{Ledrappier}.
The term ``$g$-measure" is related to the observation that $\mu_{\operatorname{TM}}$ can be constructed by tracing a (normalized) function $\widetilde{g}$, in this case given by
\[
\widetilde{g} \colon \TT \to [0,1], 
\quad \widetilde{g}(x) = \frac{1}{2} (1 - \cos(2 \pi x)),
\]
along the doubling map $T \colon x \mapsto 2x \mod 1$; see Section~\ref{SEC:Birkhoff-and-measure-estimates} for details and a formal definition of the term $g$-measure in our setting. 

The doubling map $(\TT,T)$ is closely related to the full shift $(\mathbb{X},\sigma)$, with $\mathbb{X} = \{0,1\}^\N$ and $\sigma(x)_n = x_{n+1}$ via the (inverse) binary representation $\pi_2 \colon (x_n)_{n \in \N} \mapsto \sum_{n=1}^{\infty} x_n 2^{-n}$, which semi-conjugates the action of $\sigma$ and $T$. The map $\pi_2$ is $2$-to-$1$ on the set $\mc D$ of sequences that are eventually constant (preimages of dyadic rationals), and $1$-to-$1$ everywhere else. Since the dyadic rationals are countable and hence a nullset of $\mu_{\operatorname{TM}}$, we can uniquely lift $\mu_{\operatorname{TM}}$ to a measure $\mu$ on $\mathbb{X}$ satisfying 
\[
\mu_{\operatorname{TM}} = \mu \circ \pi_2^{-1}.
\]
We adopt a standard choice for the metric on $\mathbb{X}$, given by $d(x,y) = 2^{-k+1}$ whenever $k$ is the smallest integer with $x_k \neq y_k$. We also employ for every finite word $w \in \{0,1\}^n$ and $n \in \N$ the cylinder set notation
$
[w] = \{x \in \mathbb{X} : x_{1} \cdots x_n = w_1 \cdots w_n \}.
$
The choice to work with $(\mathbb{X},\sigma)$ instead of $(\TT,T)$ is purely conventional and mostly made for the sake of a simpler exposition. All of the results presented in this section hold just the same over the torus and the proof works in the same way with a few minor adaptations.
 
The close relation between $\mu$ and $\widetilde{g}$ alluded to earlier, persists in a thermodynamic description of $\mu$. Indeed, due to a classical result by Ledrappier \cite{Ledrappier}, $\mu$ can alternatively be characterized as the unique equilibrium measure of the potential function
\[
\psi \colon \mathbb{X} \to [-\infty,\infty), \quad
x \mapsto \log \widetilde{g}(\pi_2(x)),
\]
which has a singularity at the preimages of the origin, $x = 0^{\infty}$ and $x=1^{\infty}$.
A multifractal analysis for the Birkhoff averages $b_{\psi}$ and the local dimensions $d_{\mu}$ was performed in \cite{BGKS,FSS}. There it was shown in particular that the level sets 
\[
 \bigl\{ x \in \mathbb{X} : d_{\mu}(x) = \alpha \bigr\},
 \quad \bigl\{ x \in \mathbb{X} : b_{\psi}(x) = - \log(2) \alpha \bigr\}
\]
have full Hausdorff dimension as soon as $\alpha \geqslant 2$.
This supports the idea that a superpolynomial scaling of the the TM measure (and a superlinear growth of the Birkhoff sums) is in some sense typical for the TM measure. We pursue this idea in the following.

Since the ball of radius $2^{-n}$ around $x \in \mathbb{X} $ is given by $C_n(x):=[x_1\cdots x_n]$, we may also write the local dimension of the measure $\mu$ as
\[
d_{\mu}(x) = \lim_{n \to \infty} \frac{\log \mu(C_n(x))}{ - n \log 2},
\]
provided that the limit exists. 
The equilibrium state can be expected to avoid the singularities at the preimages of the origin (which are also fixed points of the dynamics). It is therefore reasonable to expect the fastest possible decay rate for $\mu$ at these positions. 
Given $\pi_2(x) = 0$, it was already observed in \cite{GL90} (for more refined estimates see also \cite{BCEG,BG19}) that
\begin{equation*}
\lim_{n \to \infty} \frac{\log \mu(C_n(x))}{- n^2 \log 2} = 1.
\end{equation*}
The same conclusion holds in fact for $x\in \mc D$, the preimages of dyadic rationals \cite{G-thesis} (and no other points, as we will see below).
However, this is a countable set of vanishing Hausdorff dimension. It seems natural to inquire if sets of non-trivial Hausdorff dimension occur if $n^2$ is replaced by a different scaling function.

When it comes to the Birkhoff sums, choosing $x \in \mc D$ immediately gives $S_n \psi(x) = - \infty$ for large enough $n$, so we will not get a finite result for \emph{any} scaling function. However, as long as $x \notin \mc D$, we will obtain
\begin{equation*}
\liminf_{n \to \infty} \frac{-S_n \psi(x)}{n^2 \log 2} \leqslant 1,
\end{equation*}
and in this sense the fastest possible scaling for $S_n \psi$ is also given by $n^2$. We may interpolate between the linear and quadratic scaling via the scaling function $n^{\gamma}$ for some $\gamma \in (1,2)$. It turns out that the points with such an intermediate scaling have full Hausdorff dimension.

\begin{theorem}
\label{THM:n-gamma-scaling}
For each $\gamma \in (1,2)$ and $\alpha \geqslant 0$, the level sets
\[
\biggl \{ x \in \mathbb{X} : \lim_{n \to \infty} \frac{\log \mu(C_n(x))}{-n^{\gamma} \log 2} = \alpha \biggr\},
\quad \biggl\{ x \in \mathbb{X}: \lim_{n \to \infty} \frac{-S_n \psi(x)}{n^{\gamma} \log 2} = \alpha \biggr\}
\]
have Hausdorff dimension $1$.
\end{theorem}

In this sense, $n^2$ is the critical scaling, at least for phenomena that can be distinguished via Hausdorff dimension. We will therefore focus on accumulation points for this particular scaling in the following.

Although the relation between $S_n \psi(x)$ and $\mu(C_n(x))$ is not as simple as in the H\"{o}lder continuous case, their asymptotic behavior is still closely related. In fact, both expressions can be controlled via an appropriate recoding of $x \in \mathbb{X}$. As long as $x \notin \mc D$, its binary representation can be uniquely written in an alternating form as
$
x = a^{n_1} b^{n_2} a^{n_3} b^{n_4} \ldots,
$
where $a,b \in \{0,1\}$ with $a\neq b$ and $n_i \in \N$ for all $i \in \N$. With this notation, the \emph{alternation coding} is a map $\tau \colon \mathbb{X}\setminus \mc D \to \N^{\N}$, given by
\[
\tau \colon a^{n_1} b^{n_2} a^{n_3} b^{n_4} \ldots
\mapsto n_1 n_2 n_3 n_4\ldots.
\]
Given $x \in \mathbb{X} \setminus \mc D$ with $\tau(x) = (n_i)_{i \in \N}$, we define
\[
F_m(x) = \frac{1}{N_m(x)^2} \sum_{i=1}^m n_i^2,
\quad N_m(x) = \sum_{i=1}^m n_i,
\]
for all $m \in \N$. For notational convenience, we also set $\overline{F}(x) = \limsup_{m \to \infty} F_m(x)$ and $\underline{F}(x) = \liminf_{m \to \infty} F_m(x)$. The role of this sequence of functions is clarified by the following result.

\begin{prop}
\label{PROP:accumulation-gaps}
Given $x \in \mathbb{X} \setminus \mc D$, let $\underline{F}(x) = \alpha$ and $\overline{F}(x) = \beta$. 
Then,
\[
\liminf_{n \to \infty} \frac{\log \mu(C_n(x))}{- n^2 \log 2} = \frac{\alpha}{1+\alpha}, \quad
\limsup_{n \to \infty} \frac{\log \mu(C_n(x))}{- n^2 \log 2} = \beta,
\]
and 
\[
\liminf_{n \to \infty} \frac{- S_n \psi(x)}{n^2 \log 2} = \alpha, \quad
\limsup_{n \to \infty} \frac{- S_n \psi(x)}{n^2 \log 2} = \frac{\beta}{1-\beta}.
\]
\end{prop}

This has the following remarkable consequence.

\begin{coro}
Whenever the sequence $\log \mu(C_n(x))/n^2$ has a non-trivial accumulation point ($\neq 0$), the accumulation points form in fact an interval of strictly positive length. The same conclusion holds for the sequence $S_n \psi(x)/n^2$.
\end{coro}

Also, we immediately obtain a gap result for dyadic vs non-dyadic points.

\begin{coro}
If $x \in \mc D$, then $S_n \psi(x) = -\infty$ for large enough $n$, and
\[
\lim_{n \to \infty} \frac{\log\mu(C_n(x))}{n^2} = - \log 2.
\]
In contrast, if $x \in \mathbb{X} \setminus \mc D$, then
\[
\limsup_{n \to \infty} \frac{\log\mu(C_n(x))}{n^2} \geqslant - \frac{1}{2} \log 2,
\quad 
\limsup_{n \to \infty} \frac{S_n \psi(x)}{n^2} \geqslant -1.
\]
\end{coro}

Due to the pointwise relation in Proposition~\ref{PROP:accumulation-gaps}, it suffices to focus on the accumulation points of $(F_m)_{m \in \N}$. These can be analysed via the \emph{joint (dimension) spectrum} of $\underline{F}$ and $\overline{F}$, given by
\[
(\alpha,\beta) \mapsto \dim_H\{ x : \underline{F}(x) = \alpha, \overline{F}(x) = \beta \},
\]
for $(\alpha,\beta) \in \R^2$.
More generally, we calculate the Hausdorff dimension of
\[
\{ (\underline{F}, \overline{F}) \in S \} : = 
\{ x \in \mathbb{X} \setminus \mc D : (\underline{F}(x), \overline{F}(x)) \in S \},
\]
for every subset $S \in \R^2$. Since all accumulation points of $(F_m)_{m \in \N}$ are in $[0,1]$, the pair $(\underline{F},\overline{F})$ is certainly contained in
\[
\Delta:= \{ (\alpha,\beta) \in [0,1]^2: \alpha \leqslant \beta \}.
\] 
It therefore suffices to consider sets $S \subset \Delta$.
We show that the joint spectrum is given by a function $f \colon \Delta \to [0,1]$, defined on $\Delta\setminus \{(0,0) \}$ as
\begin{equation}
\label{EQ:f-alpha-beta}
f(\alpha,\beta):=
\frac{\sqrt{\alpha \beta + \beta - \alpha} - \beta}{\sqrt{\alpha \beta + \beta - \alpha} + \sqrt{\alpha \beta}},
\end{equation}
see Figure~\ref{FIG:f-alpha-beta} for an illustration. A continuous extension of $f$ to $\Delta$ is not possible, since $f$ can take arbitrary values in $[0,1]$ as we approach the origin from different directions. We define $f(0,0) :=1$, which is the most adequate choice for our application below.

\begin{figure}
\begin{center}
\includegraphics[scale=0.4]{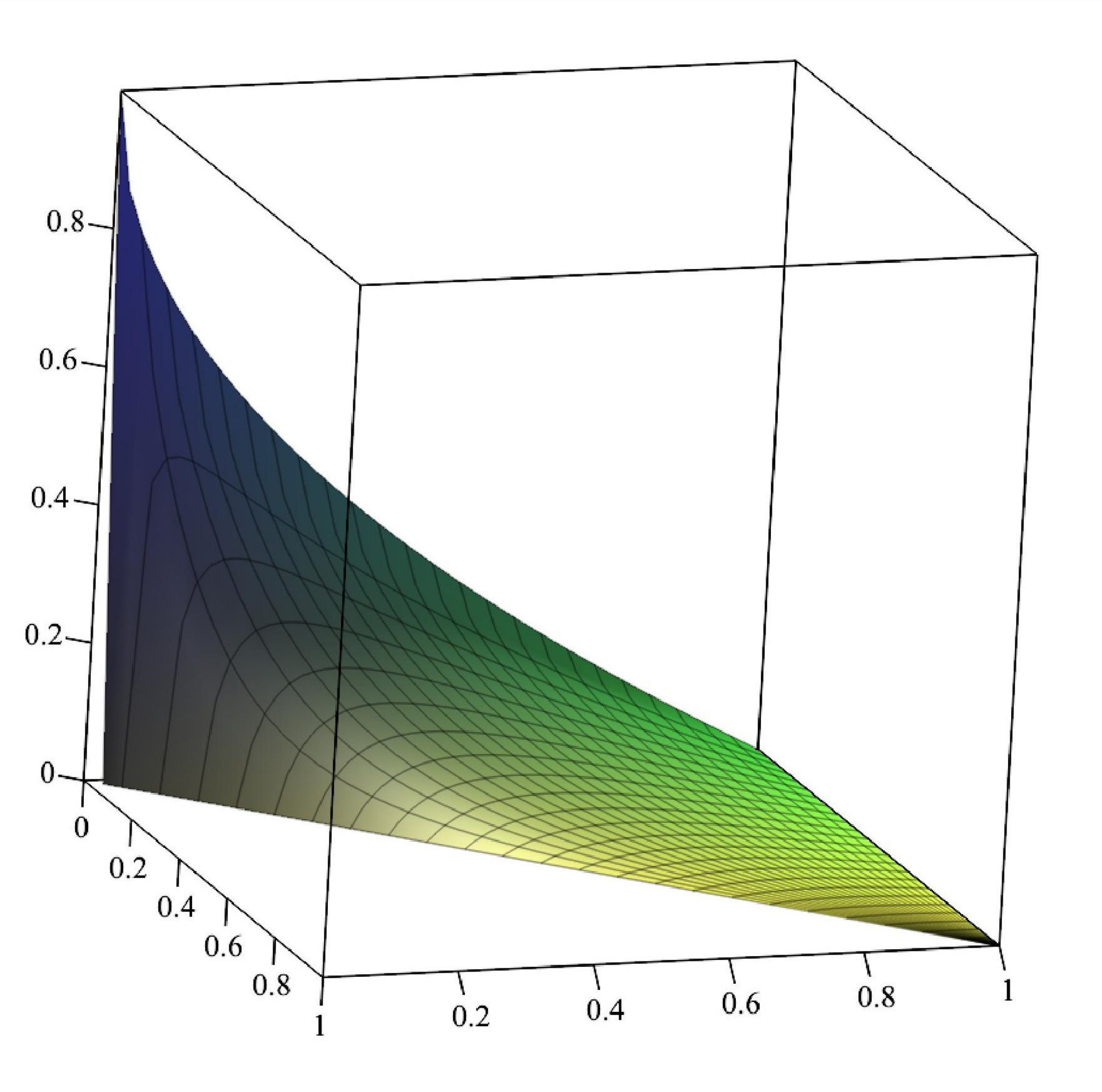}
\caption{The function $f\colon \Delta \to [0,1]$.}
\label{FIG:f-alpha-beta}
\end{center}
\end{figure}

\begin{theorem}
\label{THM:main}
Let $S \subset \Delta$. Then,
\[
\dim_H \{(\underline{F},\overline{F}) \in S \}
= \sup \{ f(\alpha,\beta) : (\alpha,\beta) \in S\}.
\]
In particular, $\dim_H\{\underline{F} = \alpha,\overline{F} = \beta \} = f(\alpha,\beta)$ for all $(\alpha,\beta) \in \Delta$.
\end{theorem}

Because of its central role, we detail some properties of the function $f$ below (without proof), which may be verified using standard tools from analysis. We describe the values of $f$ on the boundary of $\Delta$ in the first two items and proceed to monotonicity properties thereafter.

\begin{prop}
\label{PROP:f-properties}
The function $f \colon \Delta \to [0,1]$ has the following properties.
\begin{enumerate}
\item $f(\beta,\beta) = 0 = f(\alpha,1)$ for all $\beta \in (0,1]$ and $\alpha \in [0,1]$.
\item $f(0,\beta) = 1 - \sqrt{\beta}$ for all $\beta \in [0,1]$.
\item $f(\alpha,\beta) > 0$ for all $(\alpha,\beta)$ in the interior of $\Delta$.
\item The map $\alpha \mapsto f(\alpha,\beta)$ is decreasing in $\alpha$ for all $\beta$.
\item For every $\alpha \in (0,1)$, there is a value $\alpha^*$ with $\alpha< \alpha^* < 1$ such that $\beta \mapsto f(\alpha,\beta)$ is strictly increasing on $(\alpha,\alpha^*)$, takes its maximum in $\beta = \alpha^*$ and is strictly decreasing on $(\alpha^*,1)$.
\end{enumerate}
\end{prop}

Especially the last property in Proposition~\ref{PROP:f-properties} is remarkable as it shows that, for a fixed value $\underline{F} \in (0,1)$, most points (in the sense of Hausdorff dimension) achieve a value of $\overline{F}$ that lies strictly between $\underline{F}$ and $1$. 
We emphasize that, due to Proposition~\ref{PROP:accumulation-gaps}, the result in Theorem~\ref{THM:main} can also be regarded as a statement about the level sets for the $\liminf$ and $\limsup$ of the sequences $\log \mu(C_n(x)/n^2$ and $S_n \psi(x)/n^2$, respectively. In particular, the non-triviality of the joint spectrum of the $\limsup$ and the $\liminf$ persists.
Let us single out two more consequences for the reader's convenience.

\begin{coro}
Given $\beta \in [0,1]$, we have
\[
\dim_H
\biggl \{ x \in \mathbb{X} : \limsup_{n \to \infty} \frac{\log \mu(C_n(x))}{-n^2 \log 2} = \beta \biggr\} = 1 - \sqrt{\beta}.
\]
\end{coro}

\begin{coro}
The set of points $x \in \mathbb{X}$ with $\liminf_{n \to \infty} S_n \psi(x)/n^2 = - r$ has positive Hausdorff dimension if $r \in [0,\infty)$ and vanishing Hausdorff dimension if $r= \infty$.
\end{coro}

\section{Estimates for Birkhoff sums and measure decay}
\label{SEC:Birkhoff-and-measure-estimates}

We begin with a few preliminaries on notation and basic concepts. Given two (real-valued) sequences $(f_m)_{m \in \N}$ and $(g_m)_{m \in \N}$, we write $f_m \sim g_m$ if $f_m/g_m \to 1$ as $m \to \infty$. Similarly, $f_m = o(g_m)$ if $f_m/g_m \to 0$ and $f_m = O(g_m)$ if $f_m/g_m$ is bounded as $m \to \infty$.

Every Borel probability measure $\nu$ on $\mathbb{X}$ may also be regarded as a linear functional on the space of continuous functions $C(\mathbb{X})$. This motivates the notation $\nu(f):= \int f \dd \nu$ for $f \in C(\mathbb{X})$, which we sometimes extend to $\nu$-integrable functions $f$. 

Following \cite{Keane,Ledrappier}, a \emph{$g$-function} over $(\mathbb{X},\sigma)$ is a Borel measurable function $g \colon \mathbb{X} \to [0,1]$ satisfying $\sum_{y \in \sigma^{-1}x} g(y) = 1$ for all $x \in \mathbb{X}$. There is a corresponding transfer operator
\[
{\mc L}_{g} \colon C(\mathbb{X}) \to C(\mathbb{X}), 
\quad ({\mc L}_{g} f)(x) = \sum_{y \in \sigma^{-1}x} g(y) f(y).
\]
We call $\nu$ a \emph{$g$-measure} with respect to $g$ if it is invariant under the dual of ${\mc L}_g$, that is, $\nu({\mc L}_g f) = \nu(f)$ for all $f \in C(\mathbb{X})$. 
It is straightforward to check that $g = \widetilde{g} \circ \pi_2$, with $\widetilde{g}(x) = (1 - \cos(2\pi x))/2$ is indeed a $g$-function with $g$-measure $\mu$; compare \cite{BCEG} for the corresponding statement about $\widetilde{g}$ and $\mu_{\operatorname{TM}}$ over the doubling map. In fact $\mu_{\operatorname{TM}}$ is known to be the \emph{unique} $g$-measure with respect to $\widetilde{g}$. We refer to \cite{BCJOE,BK,CR,Keane} and the references therein for more on the (non-)uniqueness of $g$-measures.

Since $g = \exp \circ \psi$, the invariance of $\mu$ under ${\mc L}_{g}$ builds a natural bridge to the potential function.
This can be used to obtain the following replacement for the Gibbs property in the H\"{o}lder continuous case.

\begin{lemma}
\label{LEM:lower-bound}
For any two words $w \in \{0,1 \}^n$ and $v \in \{0,1\}^m$, we have
\[
\mu([wv]) = \int_{[v]} g_n (wx) \dd \mu(x),
\]
where 
\[
g_n(x) = \prod_{k = 0}^{n-1} g(\sigma^k x).
\]
In particular,
\[
 \inf_{x \in [wv]} S_n \psi(x) + \log (\mu[v])
 \leqslant \log(\mu[wv]) 
 \leqslant \log(\mu[w])
 \leqslant \sup_{x \in [w]} S_n \psi(x). 
\]
\end{lemma}

\begin{proof}
Writing $\mathds{1}_{[wv]}$ for the characteristic function of $[wv]$ and using the invariance of $\mu$ under the transfer operator, we get 
\[
\mu([wv]) = \mu(\mathds{1}_{[wv]}) = \mu({\mc L}_{g}^n \mathds{1}_{[wv]}),
\]
and obtain via a straightforward calculation
\[
{\mc L}_{g}^n \mathds{1}_{[wv]}\colon x \mapsto 
\sum_{w' \in \{0,1 \}^n } g_n(w' x) \mathds{1}_{[wv]}(w' x)
= g_n(wx) \mathds{1}_{[v]}(x),
\]
This yields the first assertion. The inequalities follow by estimating the integrand via its infimum (or supremum) and taking the logarithm.
\end{proof}

We continue by recording a basic estimate for the potential function. The proof is straightforward and left to the interested reader.

\begin{lemma}
\label{LEM:psi-estimate}
For every $x \in \TT$, let $|x|$ be the smallest Euclidean distance to an endpoint of the unit interval. Then, we have
\[
2 \log(2 |x|) \leqslant \log \widetilde{g}(x) \leqslant 2 \log(\pi |x|).
\]
\end{lemma}

We use these bounds to obtain an estimate for $S_n \psi(x)$ for arbitrary $n \in \N$ and $x \in \mathbb{X} \setminus \mc D$. 
Recall the notation $\tau(x) = (n_m)_{m \in \N}$ and $N_m = \sum_{i=1}^{m} n_i$ for $m \in \N$.

\begin{lemma}
\label{LEM:psi-n-bounds}
Let $N_m \leqslant n < N_{m+1}$ for some $m \in \N$ and $r_{m+1} = N_{m+1} - n > 0$. Then, 
\begin{align*}
 - \log 2 \biggl( n + \sum_{i=1}^{m+1} n_i^2 - r_{m+1}^2\biggr) \leqslant
S_n \psi(x)  \leqslant - \log 2 \biggl ( n + \sum_{i=1}^{m+1} n_i^2 - r_{m+1}^2 \biggr) + n \log \pi .
\end{align*}
\end{lemma}

\begin{proof}
First, note that if $y \in [0^k 1]$ for some $k \in \N$, then
$
2^{-(k+1)} \leqslant |\pi_2(y)| \leqslant 2^{-k},
$
which by Lemma~\ref{LEM:psi-estimate} implies that
\[
- 2k \log 2
\leqslant  \psi(y) 
\leqslant - 2k \log 2 + \log \pi.
\]
Let $k' = k-r$ for some $0\leqslant r < k$.
Since for $0 \leqslant \ell < k$ the point $\sigma^\ell y$ is contained in $[0^{k-\ell}1]$, we can estimate
\begin{equation}
\label{EQ:psi-kprime-lower}
S_{k'} \psi(y) = \sum_{\ell=0}^{k-r-1} \psi(\sigma^\ell y) 
\geqslant - 2 \log 2 \sum_{\ell=0}^{k-r-1} (k-\ell) 
= -(k^2 - r^2 + k') \log 2. 
\end{equation}
In the special case $k'= k$, this yields
\begin{equation}
\label{EQ:psi-k-bounds}
S_k \psi(y) \geqslant -(k^2 + k) \log 2. 
\end{equation}
By symmetry, the same bounds hold if $y \in [1^k 0]$. For simplicity let us assume that 
\[
x = 0^{n_1} 1^{n_2} \cdots 1^{n_m} 0^{n_{m+1}} \cdots.
\]
All other cases work analogously. Since $n+r_{m+1} = N_{m+1} = N_m + n_{m+1}$, we have in particular that $n-N_m = n_{m+1} - r_{m+1}$. Using this, we can split up the Birkhoff sum as
\begin{align*}
S_n \psi(x) &=S_{n_1} \psi(0^{n_1} 1\cdots) + \ldots + S_{n_m}\psi(1^{n_m} 0\cdots)
+ S_{n_{m+1} - r_{m+1}} \psi(0^{n_{m+1}} 1 \cdots) 
\\ &\geqslant - \log 2 \biggl( n + \sum_{i=1}^m n_i^2 + (n_{m+1}^2 - r_{m+1}^2) \biggr),
\end{align*}
using \eqref{EQ:psi-kprime-lower} and \eqref{EQ:psi-k-bounds} in the last step. This shows the lower bound. The upper bound follows along the same lines.
\end{proof}

Although $\mu(C_n(x))$ is closely related to $S_n \psi(x)$ via Lemma~\ref{LEM:lower-bound}, we emphasize that, in contrast to $S_n \psi(x)$, the expression $\mu(C_n(x))$ depends only on the first $n$ positions of $x$.
To account for this fact, we extend the action of the alternation coding $\tau$ to finite words via
\[
\tau\colon a^{n_1} b^{n_2} \cdots a^{n_m} \mapsto n_1 \cdots n_m,
\]
for $a\neq b$, (and $m$ odd) and accordingly if the word ends in $b^{n_m}$ (if $m$ is even).

\begin{lemma}
\label{LEM:square-sum-bound}
Let $w \in \{0,1\}^n$ with $\tau(w) = (n_1,\ldots,n_m)$. Then,
\[
- \biggl( n + 1 + \sum_{i=1}^m n_i^2 \biggr) \log 2
\leqslant \log \mu([w])
 \leqslant  - \biggl(n + \sum_{i = 1}^m n_i^2 \biggr) \log 2 + n \log \pi.
\]
\end{lemma}

\begin{proof}
Again, it suffices to consider the case that $w$ is of the form \[
w = 0^{n_1} 1^{n_2} \cdots 0^{n_{m-1}} 1^{n_m}.
\]
From Lemma~\ref{LEM:lower-bound} (and using $\mu[0] = 1/2$ by symmetry considerations), we obtain
\begin{equation}
\label{EQ:measure-bound-inf}
\inf_{x \in [w0]} S_n \psi(x) - \log 2
\leqslant \log \mu[w0]
\leqslant \log \mu[w]
\leqslant \sup_{x \in [w]} S_n \psi(x).
\end{equation}
For the lower bound, we can apply Lemma~\ref{LEM:psi-n-bounds} to $x \in [w0]$ with $n=N_m$ and $r_{m+1} = n_{m+1}$ which immediately gives the desired estimate.
For the upper bound, assume that $x \in [w]$ and note that in this case, $\tau(x)$ is of the form
\[
\tau(x) = n_1 \cdots n_{m-1} n_m(x) \cdots,
\]
with $n_m(x) \geqslant n_m$ and $N_{m-1} < n \leqslant N_m$. If $n = N_m$, we have $n_m = n_m(x)$ and may argue as for the lower bound. We hence assume $N_{m-1} < n < N_m$ in the following. Then, $r_m = N_{m} - n$ is equal to $n_m(x) - n_m$ by construction. From this, we easily conclude that $n_m^2 \leqslant n_m(x)^2 - r_m^2$. Combining this estimate with the upper bound provided by Lemma~\ref{LEM:psi-n-bounds} yields
\[
S_n \psi(x) \leqslant - \biggl(n + \sum_{j = 1}^m n_j^2 \biggr) \log 2 + n \log \pi.
\]
Since $x \in [w]$ was arbitrary, this concludes the proof via \eqref{EQ:measure-bound-inf}.
\end{proof}

We summarize our findings in terms of the function sequence $(f_m)_{m \in \N}$, with
\[
f_m(x) = \sum_{i=1}^m n_i^2,
\] 
for all $x \in \mathbb{X} \setminus \mc D$ and $m \in \N$. 
For an illustration of the following proposition we refer to Figure~\ref{FIG:square-interpolation}.

\begin{prop}
\label{PROP:psi-n-log-mu-bounds}
Let $N_m \leqslant n < N_{m+1}$ for some $m \in \N$, with $r_{m+1} = N_{m+1} - n$ and $s_{m+1} = n - N_m$. Then,
\begin{align*}
S_n \psi(x) &= - (f_{m+1}(x) - r^2_{m+1})\log 2  + O(n),
\\ \log \mu(C_n(x)) &= - (f_{m}(x) + s^2_{m+1})\log 2 + O(n).
\end{align*}
\end{prop}

\begin{figure}
\begin{tikzpicture}
\draw[->] (-1.5,-1)--(-1.5,4);
\draw[->] (-1.5, -1) -- (5,-1) node[right, outer sep=0.2cm] {$n$};
  \draw[domain=0:3, samples = 300, variable=\x] plot ({\x}, {\x^2/3 } );
  \draw[domain=0:3, samples = 300, variable=\x, dashed] plot ({\x}, {3 - (\x - 3)^2/3 } );

\draw[dotted] (0,0) -- (-1.5,0)  node[left] {$f_m(x)$};
\draw[dotted] (3,3) -- (-1.5,3) node[left] {$f_{m+1}(x)$};

\draw[dotted] (0,0) -- (0,-1) node[below] {$N_m$};
\draw[dotted] (3,3) -- (3,-1) node[below] {$N_{m+1}$};
\end{tikzpicture}
\caption{Estimates (up to $O(n)$) for $-\log \mu(C_n(x))/\log 2$ (solid) and for $-S_n \psi(x)/\log 2$ (dashed), given in Proposition~\ref{PROP:psi-n-log-mu-bounds}.}
\label{FIG:square-interpolation}
\end{figure}
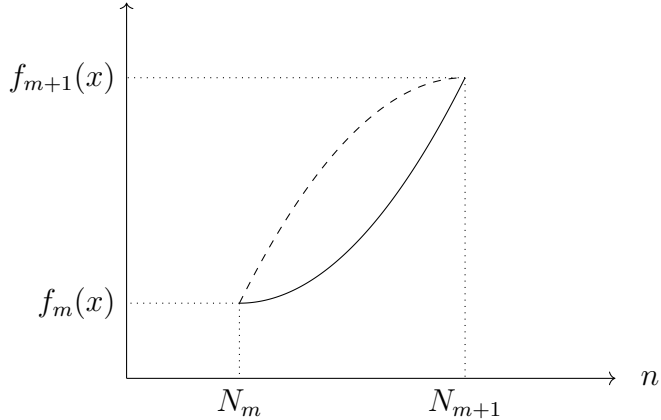

\begin{remark}
It is worth noticing that both $N_m(x)$ and $f_m(x)$ are themselves Birkhoff sums over $(\N^\N,\sigma)$. More precisely,  $N_m(x) = S_m \varphi(\tau(x))$, with $\varphi \colon n_1 n_2 \ldots \mapsto n_1$ and  $f_m(x) = S_m \varphi^2(\tau(x))$, where $\varphi^2 \colon n_1 n_2 \ldots \mapsto n_1^2$. Hence, we are in fact concerned with locally constant, unbounded observables over the full shift with a countable alphabet.
\end{remark}


\section{Intermediate scaling}
\label{SEC:intermediate}

In this Section we investigate the scaling function $n \mapsto n^{\gamma}$ for $\gamma \in (1,2)$ and prove that this scaling is typical for $S_n \psi(x)$ and $\log \mu(C_n(x))$ in the sense of full Hausdorff dimension. As a first step, we show that we may restrict our attention to the limiting behavior of $f_m$ as $m \to \infty$.

\begin{lemma}
\label{LEM:gamma-limits}
Assume that $x \in \mathbb{X} \setminus \mc D$ and $\lim_{m\to\infty}N_m^{-\gamma} f_m(x) = \alpha > 0$. Then,
\[
\lim_{n \to \infty} \frac{\log \mu(C_n(x))}{n^{\gamma}} = \lim_{n \to \infty} \frac{S_n \psi(x)}{n^{\gamma}} = - \alpha \log 2.
\]
\end{lemma}

\begin{proof}
First, we will show that the convergence of $N_m^{-\gamma} f_m(x)$ implies that both $n_m/N_m$ and $n_m^2/N_m^{\gamma}$ converge to $0$. Indeed, whenever $n_m/N_m > \delta > 0$, we get $f_m(x) \geqslant \delta^2 N_m^2$, which can happen only for finitely many values of $m$. This implies also $\lim_{m \to \infty} N_m/N_{m+1} =1$. Finally, if $n_m^2/ N_m^{\gamma} > \delta >0$ for infinitely many $m$, we obtain
\[
\frac{f_{m}(x)}{N_{m}^{\gamma}}
=  \frac{f_{m-1}}{N_m^{\gamma}} + \frac{n_{m}^2}{N_m^{\gamma}},
\]
and applying the $\limsup$ to both sides yields $\alpha \geqslant \alpha + \delta$, a contradiction.
These observations offer enough control over the points $N_m \leqslant n < N_{m+1}$ to obtain the desired convergence from Proposition~\ref{PROP:psi-n-log-mu-bounds} (and the fact that $0 \leqslant r_m,s_m \leqslant n_m$ in the corresponding notation).
\end{proof}

In order to establish lower bounds for the Hausdorff dimension of level sets, we will make use of the following simple consequence of the mass distribution principle.
Recall that we define the upper density of a subset $M \subset \N$ via
\[
\overline{D}(M) = \limsup_{n \to \infty} \frac{1}{n} \# (M \cap [1,n]).
\]

\begin{lemma}
\label{LEM:dimH-density}
For $M \subset \N$ and $w \colon M \to \{0,1\}$ let
\[
A = A(w) = \{x \in \mathbb{X} \, : \, x_m = w_m \mbox{ for all } m \in M \}.
\]
Then, $\dim_H A \geqslant 1 - \overline{D}(M)$.
\end{lemma}

\begin{proof}
We define a Bernoulli-like measure $\nu$ on $A$ by ``ignoring the determined positions". More precisely, for every $n \in \N$ let $P_n = \{1,\ldots,n\} \setminus M$ be the free positions and set $c_n = \# P_n$. Clearly, there are $2^{c_n}$ choices for $v \in \{0,1\}^n$ such that $[v]$ intersects $A$ and we set
\[
\nu [v] = \begin{cases}
2^{-c_n} & \mbox{if } [v] \cap A \neq \varnothing,
\\ 0  &\mbox{otherwise}.
\end{cases}
\]
It is straightforward to check that this definition is consistent and there is a unique measure $\nu$ with this property by the Kolmogorov extension theorem. We obtain for every $x \in A$ and $n \in \N$ that $\nu(C_n(x)) = 2^{-c_n}$ and therefore the lower local dimension of $\nu$ at $x$ is given by
\[
\underline{d}_{\nu}(x) = \liminf_{n \to \infty} \frac{\log \nu(C_n(x))}{-n \log 2} = \liminf_{n \to \infty} \frac{c_n}{n} = 1 - \overline{D}(M).
\]
The claim hence follows via the (non-uniform) mass distribution principle.
\end{proof}

With the help of Lemma~\ref{LEM:dimH-density}, we will show that for every $\beta > 0$, the situation in Lemma~\ref{LEM:gamma-limits} is typical in the sense of full Hausdorff dimension.

\begin{prop}
\label{PROP:n-gamma-full}
For every $\gamma \in (1,2)$ and $\alpha > 0$, we have 
\[
\dim_H \{ x \in \mathbb{X} \setminus \mc D : f_m(x) \sim \alpha N_m^{\gamma} \} = 1.
\]
\end{prop}

\begin{proof}
We construct a subset with Hausdorff dimension arbitrarily close to $1$. The dimension estimate will be provided by Lemma~\ref{LEM:dimH-density}. Hence, we want to find a subset $M \in \N$ of arbitrarily small upper density, such that fixing $x$ on $M$ in an appropriate way ensures that $f_m(x) \sim \alpha N_m^{\gamma}$. 
The general strategy is the following: We choose a sequence $(\theta_k)_{k \in \N}$ of positive real numbers such that $\theta_{k} - \theta_{k-1} \to \infty$ but $\theta_{k-1}/\theta_k \to 1$ as $k \to \infty$. To ensure $f_m(x) \sim \alpha N_m^{\gamma}$, we fix $x \in \mathbb{X}$ to be constant on an interval of some appropriate length $c_k$ in $[\theta_{k},\theta_{k+1}]$, and to have bounded alternation blocks outside of these intervals. Using that $c_k$ grows slower than $\theta_{k+1} - \theta_k$, this will fix $x$ on a set of positions with arbitrarily small density. The details follow.

For definiteness, we fix some large number $r = r(\gamma)$ (the exact value will be determined later) and set $\theta_k = k^r$. For $r > 1$ this satisfies $\theta_{k} - \theta_{k-1} \to \infty$ and $\theta_{k-1}/\theta_k \to 1$ for $k \to \infty$, as required. An appropriate choice of $c_k$ turns out to be
\begin{equation}
\label{EQ:c-k}
c_k = \sqrt{r \gamma \alpha} k^{\delta},
\quad \delta = \frac{r \gamma - 1}{2}.
\end{equation}
For this to grow slower than $\theta_k - \theta_{k-1}$, we require that $\delta < r-1$. Since 
\[
\frac{\delta}{r-1} = \frac{r \gamma - 1}{2 r - 2} \xrightarrow{r \to \infty} \frac{\gamma}{2} < 1,
\]
this holds true for large enough $r$ and we take some $r = r(\gamma) > 2$ with this property.
Hence, we can choose $k_0 \in \N$ such that $c_k < \theta_{k} - \theta_{k-1}$ for all $k \geqslant k_0$. We specify a set of positions via
\[
M_1 = \bigcup_{k \geqslant k_0} \{ n \in \N :  \theta_{k} - c_k \leqslant n \leqslant \theta_{k} \},
\]
and define 
\[
Q = \{ x \in \mathbb{X}\setminus \mc D: x_n = 0 \; \mbox{ for all } \; n \in M_1\}.
\]
To avoid large blocks outside of $M_1$ we further fix a large cutoff-value $\Lambda \in \N$ and set
\[
R_{\Lambda} = \{ x \in \mathbb{X}\setminus \mc D: x_n x_{n+1} = 10 \; \mbox{ for all } \; n \in \Lambda \N \setminus M_1 \}.
\]
Finally, we combine both conditions by setting
\[
A_{\Lambda} = Q \cap R_{\Lambda}.
\]
Given $x \in A_{\Lambda}$, we want to show that $f_m(x) \sim \alpha N_m^{\gamma}$. The definition of $Q$ implies that $x$ is constant on $[\theta_{k} - c_k, \theta_k]$ for each $k \in \N$. If $\tau(x) = (n_i)_{i \in \N}$ is the alternation coding of $x$, this implies that for every $k$ there is a unique index $i_k$ such that $N_{i_k-1} \leqslant \lceil \theta_{k} - c_k \rceil \leqslant \lfloor \theta_{k} \rfloor \leqslant N_{i_k}$, and in particular $n_{i_k} \geqslant c_k - 2$. Since $R_{\Lambda}$ restricts the length of blocks outside of $M_1$, we find that $n_{i_k}$ can in fact not be much larger and hence
\[
n_{i_k} = c_k + O(1),
\]
where the implied constant depends on $\Lambda$. 
For all other indices $i$ we have that $n_i \leqslant \Lambda$ is bounded by a constant. Hence, for $i_k \leqslant m < i_{k+1}$ we obtain
\[
f_m(x) = \sum_{i=1}^m n_i^2 = \sum_{\ell=1}^k n_{i_\ell}^2 + O(m) 
\sim \sum_{\ell=1}^k c_\ell^2.
\]
With the specific choice of $c_k$ in \eqref{EQ:c-k}, we obtain
\[
\sum_{\ell=1}^k c_\ell^2 = \alpha \sum_{\ell=1}^k r \gamma  \ell^{r \gamma -1} \sim \alpha k^{r\gamma},
\]
using an integral estimate in the last equation.
Since $\theta_k \sim \theta_{k+1}$ and by the monotonicity of $N_m$, we also observe that $N_m \sim \theta_k = k^{r}$ for $i_k \leqslant m < i_{k+1}$, and therefore
\[
f_m(x) \sim \alpha k^{r \gamma} \sim \alpha N_m^{\gamma},
\]
as required. That is, $A_{\Lambda} \subset \{ f_m \sim \alpha N_m^{\gamma} \}$ for every $\Lambda \in \N$ and it suffices to find an appropriate lower bound for the Hausdorff dimension of $A_\Lambda$. Since the positions in $M_1$ are accumulated to the left of the values $\theta_k$, we obtain
\[
\overline{D}(M_1) = \limsup_{k \to \infty} \frac{1}{\theta_k} \sum_{\ell = 1}^{k} c_{\ell} 
= \limsup_{k \to \infty} \frac{1}{k^{r \gamma}} \sqrt{r \gamma \alpha} \frac{k^{\delta+1}}{\delta+1} = 0,
\]
using that $\delta + 1 < r \gamma$ in the last step. Because the points in $A_{\Lambda}$ are fixed on the positions given by $M_1 \cup \Lambda \N \cup  (\Lambda\N + 1)$, we obtain by Lemma~\ref{LEM:dimH-density},
\[
 \dim_H A_{\Lambda} \geqslant 1 - \overline{D}\bigl(M_1 \cup \Lambda \N \cup  (\Lambda\N + 1) \bigr) = 1 - \frac{2}{\Lambda}.
\]
Since this is a lower bound for $\dim_H\{f_m \sim \alpha N_m^{\gamma} \}$ and $\Lambda \in \N$ was arbitrary, the claim follows.
\end{proof}

\begin{proof}[Proof of Theorem~\ref{THM:n-gamma-scaling}]
For $\alpha > 0$, the desired relation follows by combining Proposition~\ref{PROP:n-gamma-full} with Lemma~\ref{LEM:gamma-limits}. For $\alpha = 0$, simply recall that both $S_n \psi(x)$ and $ \log \mu(C_n(x))$ scale linearly with $n$ for a set of full Hausdorff dimension \cite{BGKS}.
\end{proof}


\section{Spreading of accumulation points}
\label{SEC:spreading}

We specialize to the scaling function $n \mapsto n^2$ for the remainder of this article.
We continue with the standing assumption that $x \in \mathbb{X} \setminus \mc D$.
By Proposition~\ref{PROP:psi-n-log-mu-bounds}, the accumulation points for $-\log \mu(C_n(x))/(n^2 \log 2)$ are the same as those of 
\[
\xi_n^{\mu}(x):= \frac{1}{n^2} \biggl(\sum_{i=1}^m n_i^2 + (n-N_m)^2 \biggr), \mbox{ if } N_m \leqslant n < N_{m+1}.
\]
Similarly, the accumulation points of $-S_n \psi(x)/(n^2 \log 2)$ coincide with those of
\[
\xi_n^{\psi}(x):= \frac{1}{n^2} \biggl(\sum_{i=1}^{m} n_i^2 - (N_m-n)^2 \biggr), \mbox{ if } N_{m-1} \leqslant n < N_{m}.
\]
Recall the notation $F_m(x) = N_m^{-2} f_m(x)$, together with $\underline{F}(x) = \liminf_{m \to \infty} F_m(x)$ and $\overline{F}(x) = \limsup_{m \to \infty} F_m(x)$.
The strict convexity of the function $s \mapsto s^2$ causes the sequence $\xi_n^{\mu}(x)$ to take its minimum  on $[N_m,N_{m+1}]$ at some intermediate point, provided that $n_{m+1}$ is sufficiently large; compare Figure~\ref{FIG:square-interpolation}. This gives rise to a drop of the $\liminf$, as compared to $\underline{F}(x)$. 

\begin{lemma}
\label{LEM:h-mu-f-relation}
Given $\underline{F}(x) = \alpha$ and $\overline{F}(x) = \beta$, we have
\[
\liminf_{n \to \infty} \xi_n^{\mu}(x) = \frac{\alpha}{1 + \alpha}, \quad
\limsup_{n \to \infty} \xi_n^{\mu}(x) = \beta.
\]
\end{lemma}

\begin{proof}
We start with the assertion about the $\liminf$.
Let $m \in \N$ and assume that $n = (1+ c) N_m$ (not necessarily $n < N_{m+1}$) for some $c\geqslant 0$. We obtain
\[
n^2 \xi_n^{\mu}(x) \leqslant \sum_{i=1}^{m} n_i^2 + (c N_m)^2 = N_m^2 (F_m(x) + c^2),
\]
with equality if and only if $n \leqslant N_{m+1}$.
Hence,
\begin{equation}
\label{EQ:f_n-in-between-bound}
\xi_n^{\mu}(x) \leqslant \frac{F_m(x) + c^2}{(1+c)^2},
\end{equation}
again with equality precisely if $n \leqslant N_{m+1}$.
For $r > 0$, the function
\[
\phi_{r} \colon c \mapsto  \frac{r + c^2}{(1+c)^2}
\]
is strictly decreasing on $[0,r)$, takes a minimum in $c = r$ and is increasing for $c \geqslant r$. This yields for $N_m \leqslant n \leqslant N_{m+1}$,
\[
\xi_n^{\mu}(x) \geqslant \min_{c > 0} \phi_{F^{}_m(x)}(c)
=  \frac{F_m(x)}{1 + F_m(x)},
\]
and in particular,
\[
\liminf_{n\to \infty} \xi_n^{\mu}(x) \geqslant
\liminf_{m \to \infty} \frac{F_m(x)}{1 + F_m(x)}
= \frac{\alpha}{1 + \alpha}.
\]
On the other hand, let 
\[
c_m = \frac{\lfloor{ F_m(x) N_m \rfloor}}{N_m},
\]
and note that if $F_{m_k}(x)$ converges to $\alpha$, then so does $c_{m_k}$ as $k \to \infty$. In particular, we find for $r_k = N_{m_k} (1 + c_{m_k})$ that
\[
\xi^{\mu}_{r_k}(x) 
\leqslant \frac{F_{m_k}(x) + c_{m_k}^2}{(1 + c_{m_k})^2}
\xrightarrow{k \to \infty} 
\frac{\alpha}{1 + \alpha},
\]
and the claim on the $\liminf$ follows.

For $n = (1+ c)N_m$ let $I$ be the interval of values $c$ such that $N_m \leqslant n \leqslant N_{m+1}$. Due to the monotonicity properties of $c \mapsto \phi_{F^{}_m(x)}(c)$, its maximum on $I$ is obtained on a boundary point. By \eqref{EQ:f_n-in-between-bound}, we hence conclude that $\xi_n^{\mu}(x) \leqslant F_m(x)$ or $\xi_n^{\mu}(x) \leqslant F_{m+1}(x)$, with equality if $n = N_m$ or $n = N_{m+1}$, respectively.
This implies the assertion about the $\limsup$. 
\end{proof}

\begin{lemma}
\label{LEM:h-psi-f-relation}
Given $\underline{F}(x) = \alpha$ and $\overline{F}(x) = \beta$, we have
\[
\liminf_{n \to \infty} \xi_n^{\psi}(x) = \alpha, \quad
\limsup_{n \to \infty} \xi_n^{\psi}(x) = \frac{\beta}{1-\beta}.
\]
\end{lemma}

\begin{proof}
This is similar to the proof of Lemma~\ref{LEM:h-mu-f-relation}. With $m \in \N$ and $n = (1-c)N_m$ for some $0 \leqslant c < 1$, we get
\[
n^2 \xi_n^{\psi}(x) \geqslant \sum_{i=1}^m n_i^2 - (cN_m)^2 = N_m^2 (F_m(x) - c^2),
\]
with equality if and only if $n \geqslant N_{m-1}$. Therefore,
\[
\xi_n^{\psi}(x) \geqslant \frac{F_m(x) - c^2}{(1-c)^2} =: \bar{\phi}_{F_m(x)}(c),
\]
again with equality if and only if  $n \geqslant N_{m-1}$.
For $0 < r < 1$, the function $\bar{\phi}_{r}(c)$ is strictly increasing on $[0,r)$, takes a maximum in $c = r$ and is decreasing for $c \in (r,1)$. Noting that $\bar{\phi}_{r}(r) = r/(1-r)$, the rest follows precisely as in the proof of Lemma~\ref{LEM:h-mu-f-relation}.
\end{proof}

\begin{proof}[Proof of Proposition~\ref{PROP:accumulation-gaps}] The corresponding statements for the accumulation points of $(\xi^{\mu}_n(x))_{n \in \N}$ and $(\xi^{\psi}_n(x))_{n \in \N}$ are given in Lemma~\ref{LEM:h-mu-f-relation} and Lemma~\ref{LEM:h-psi-f-relation}. Combining this with Proposition~\ref{PROP:psi-n-log-mu-bounds} gives the desired relations for the Birkhoff sums and the measure decay.
\end{proof}




\section{Lower bounds}
\label{SEC:lower-bound}

We want to establish necessary and sufficient criteria for $x$ to satisfy $\underline{F}(x) = \alpha$ and $\overline{F}(x) = \beta$.
We show that this requires a certain number of large blocks in the alternation coding $\tau(x) =(n_i)_{i \in \N}$.
To be more precise, let us start with a certain large cutoff-value $\Lambda \in \N$ and let
\[
I_{\Lambda} = \{ i \in \N : n_i \geqslant \Lambda \}.
\]
It will be convenient to ignore all contributions of  $n_i$ to $F_n(x)$ as long as $n_i < \Lambda$. This is achieved by setting 
\[
F^{\Lambda}_m(x) = \frac{1}{N_m^2} \sum_{i \in I_\Lambda \cap [1,m]} n_i^2.
\]

\begin{lemma}
\label{LEM:F-Delta-convergence}
We have $|F_m(x) - F^\Lambda_m(x)| \in O(N_m^{-1})$. Hence, $(F^\Lambda_m(x))_{m \in \N}$ and $(F_m(x))_{m \in \N}$ have the same set of accumulation points.
\end{lemma}
\begin{proof}
This follows by
\[
|F_m(x) - F^\Lambda_m(x)| = \frac{1}{N_m^2} \sum_{i \in [1,m]\setminus I_\Lambda} n_i^2 <
\frac{1}{N_m^2} m \Lambda^2 \leqslant \frac{\Lambda^2}{N_m},
\]
which gives the desired estimate.
\end{proof}

In principle, it is possible that $F_m^{\Lambda}(x) = 0$ for all $m \in \N$. However, this can only happen if $\overline{F}(x) = 0$, a case that we will treat separately. In the following, we always assume that $m$ is large enough to ensure $F_m^\Lambda(x) > 0$.

If $n_{j+1} < \Lambda$, we interpolate $F_r^\Lambda(x)$ continuously between $r = j$ and $r = j+1$ by setting
$N_r = N_j + (r-j) n_{j+1}$ and
\[
F^\Lambda_r(x) = \frac{1}{N_r^2} \sum_{i \in I_\Lambda \cap [1,j]} n_i^2.
\]
For a lower bound on the Hausdorff dimension of $\{ \underline{F} = \alpha, \overline{F} = \beta \}$, we wish to provide a mechanism that produces an abundance of points with this property. More precisely, we exhibit a subset of $\{ \underline{F} = \alpha, \overline{F} = \beta \}$ that permits a lower estimate for its the Hausdorff dimension via Lemma~\ref{LEM:dimH-density}.  The main idea is the following: Given $r \in \R$ with $F_r^\Lambda(x) = \beta$ we introduce blocks of length smaller than $\Lambda$ until we hit the level $F_{k-1}^{\Lambda}(x) = \alpha$ for some $k \in \N$. Since these blocks can be chosen arbitrarily we interpret them as degrees of freedom or ``undetermined positions".
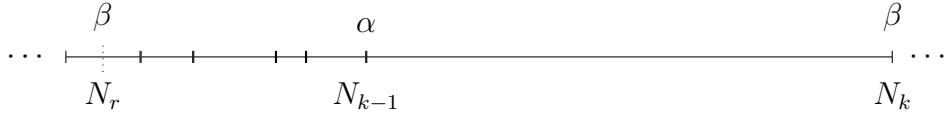
\begin{figure}
\begin{tikzpicture}
\path (0,0) --node{$\cdots$} (1,0);
\draw[|-|] (1,0) -- (2,0);
\draw[dotted] (1.5, -0.2) -- (1.5,0.2);
\draw (1.5,0) node[anchor=north, outer sep=0.2cm] {$N_r$};
\draw (1.5,0) node[anchor = south, outer sep = 0.2cm] {$\beta$};
\draw[|-|] (2,0) -- (2.7,0);
\draw[|-|] (2.7,0) -- (3.8,0);
\draw[|-|] (3.8,0) -- (4.2,0);
\draw[|-|] (4.2,0) -- (5,0);
\draw (5,0) node[anchor = north, outer sep=0.2cm] {$N_{k-1}$};
\draw (5,0) node[anchor = south, outer sep = 0.2cm] {$\alpha$};
\draw[|-|] (5,0) -- (12,0);
\draw (12,0) node[anchor = north, outer sep=0.2cm] {$N_{k}$};
\draw (12,0) node[anchor = south, outer sep = 0.2cm] {$\beta$};
\path (12,0) --node{$\cdots$} (13,0);
\end{tikzpicture}
\caption{Example for the alternation block decomposition of $x$, given that $F^\Lambda_r(x) = F^\Lambda_k(x) = \beta$ and $F^\Lambda_{k-1}(x) = \alpha$. All blocks between $N_r$ and $N_{k-1}$ have length below $\Lambda$. }
\label{FIG:degrees-of-freedom}
\end{figure}
We then add a single large block of size $n_k$ (the ``determined positions") that raises the level back to $F_k^{\Lambda}(x) = \beta$; compare Figure~\ref{FIG:degrees-of-freedom} for an illustration. 
The relative amount $f(\alpha,\beta)$ of undetermined positions turns out to be independent of the starting position $r$. Repeating this procedure, the lower density of undetermined positions equals $f(\alpha,\beta)$ over the whole sequence. This will yield the same value as a lower bound on the Hausdorff dimension of $\{ \underline{F} = \alpha, \overline{F} = \beta \}$. In Section~\ref{SEC:upper-bound}, we will prove that this strategy is indeed optimal, establishing $f(\alpha,\beta)$ also as an upper bound for the Hausdorff dimension.

\begin{lemma}
\label{LEM:f-beta-alpha-dof}
Let $j,k \in \N$ with $j<k$ and assume that $n_i < \Lambda$ for all $j < i < k$ and $n_{k} \geqslant \Lambda$. Suppose that there is $j \leqslant r < j+1$ such that $F^\Lambda_r(x) = F^\Lambda_k(x) =:\beta$ and set $\alpha: = F^\Lambda_{k-1}(x)$. Then,
\[
\frac{N_{k-1} - N_r}{N_k - N_r} = f(\alpha,\beta),
\]
with $f(\alpha,\beta)$ as defined in \eqref{EQ:f-alpha-beta}.
\end{lemma}

\begin{proof}
Let $N_{k-1} = (1+\ell) N_r$ and $N_k = (1+\ell+m) N_r$. Since $[j+1,k-1] \cap I_\Lambda = \varnothing$, we have
\[
N_{k-1}^2 F_{k-1}^\Lambda(x) = N_r^2 F_r^\Lambda(x),
\]
which translates to
\[
\alpha (1+\ell)^2 = \beta.
\]
Solving for $\ell$, we obtain
\[
\ell = \frac{\sqrt{\beta} - \sqrt{\alpha}}{\sqrt{\alpha}}.
\]
On the other hand, we have $n_k = m N_r$ by definition, yielding 
\[
N_k^2 F_k^\Lambda(x) = n_{k}^2 + \sum_{i \in I_\Lambda\cap[1,k-1]} n_i^2 = m^2 N_r^2 + N_r^2 F_r^\Lambda(x). 
\]
That is,
\[
 m^2 + \beta = \beta (1+\ell + m)^2 = \beta \biggl( \frac{\sqrt{\beta}}{\sqrt{\alpha}} + m \biggr)^2,
\]
which gives after a few steps of calculation,
\[
m = \frac{\sqrt{\beta}}{\sqrt{\alpha}} \frac{1}{1-\beta} \Bigl( \beta + \sqrt{\alpha \beta + \beta - \alpha} \Bigr).
\]
Finally, this implies
\begin{align*}
\frac{N_{k-1} - N_r}{N_k - N_r}
& = \frac{\ell}{\ell+m} 
= \frac{1}{1 + m/ \ell} 
 = \biggl( 1 + \frac{\sqrt{\beta}}{\sqrt{\beta} - \sqrt{\alpha}} \frac{1}{1 - \beta} \Bigl( \beta + \sqrt{\alpha\beta + \beta - \alpha} \Bigr) \biggr)^{-1}.
\end{align*}
A few formal manipulations show that this is precisely the expression given by $f(\alpha,\beta)$.
\end{proof}

In order to show that the strategy sketched before Lemma~\ref{LEM:f-beta-alpha-dof} is in a certain sense optimal, we move away from the assumption that there are only negligible blocks between $N_r$ and $N_{k-1}$. In this more general setting, we find the following analogue of Lemma~\ref{LEM:f-beta-alpha-dof} which will be useful in Section~\ref{SEC:upper-bound}.

\begin{lemma}
\label{LEM:f-dof-generalized}
Suppose $n_{k} \geqslant \Lambda$ for some $k \in \N$ and let $F^\Lambda_{k-1}(x) = \alpha < F^\Lambda_{k}(x) = \beta$. Then,
\[
\frac{N_{k-1} - \sqrt{\alpha/\beta} N_{k-1}}{N_k - \sqrt{\alpha/\beta} N_{k-1}} = f(\alpha,\beta).
\]
\end{lemma}
\begin{proof}[Sketch of proof]
The proof of Lemma~\ref{LEM:f-beta-alpha-dof} carries over verbatim if we replace $N_r$ by the term $\sqrt{\alpha/\beta} N_{k-1}$ and use the identification $F^M_r(x) = F^M_k(x)$.
\end{proof}

We can now provide a lower estimate for the dimension of the set $\{ \underline{F} = \alpha, \overline{F} = \beta\}$. 
As in Section~\ref{SEC:intermediate}, we will make use of Lemma~\ref{LEM:dimH-density}, by fixing the values of the sequence $x = (x_n)_{n \in \N}$ on an appropriate subset of $\N$.
\begin{prop} \label{PROP:lower-bound}
Let $(\alpha,\beta) \in \Delta$. Then, $\dim_H \{ \underline{F} = \alpha, \overline{F} = \beta \} \geqslant f(\alpha,\beta)$.
\end{prop}

\begin{proof}
For $\beta =0$, note that whenever the size of alternation blocks in $\tau(x)$ is uniformly bounded, it follows that $\overline{F}(x) = 0$. Since the union of all such elements $x$ has full Hausdorff dimension, the claim holds in this particular case. Likewise, the claim is trivial if $\beta = 1$ or $\alpha = \beta \neq 0$ because this implies $f(\alpha,\beta) = 0$. We can hence assume $\alpha < \beta < 1$ in the following.
For simplicity, we further restrict to the case that $\alpha > 0$. The case $\alpha = 0$ can be treated by replacing $\alpha$ with a sequence $\alpha_k \to 0$ in the argument below.

We follow the ideas outlined before Lemma~\ref{LEM:f-beta-alpha-dof}, using some of the notation introduced in its proof.
For $\Lambda \in \N$, we specify a set of positions, given by
	\[
	M_\Lambda := \bigcup_{k \geqslant 0}\{n \in \N: \ (1+\ell)\theta_k \leq n \leq \theta_{k+1} \},
	\]
	where $\theta_k = (1+\ell+m)^k \theta_0$ for all $k\in \N_0$, $\ell = \frac{\sqrt{\beta} - \sqrt{\alpha}}{\sqrt{\alpha}}$, $m=\frac{\sqrt{\beta}}{\sqrt{\alpha}} \frac{1}{1-\beta}(\beta + \sqrt{\alpha \beta + \beta-\alpha})$ and $\theta_0 \in \N$ is a value with $m\theta_0 > \Lambda + 2$. 
We recall from the proof of Lemma~\ref{LEM:f-beta-alpha-dof} that
	\begin{equation} \label{EQ:l-m-alpha-relations}
	m^2 + \beta = \beta (1+\ell + m)^2,
	\quad \frac{\ell}{m+\ell} = f(\alpha,\beta).
	\end{equation}
The set $M$ will denote those positions where the binary expansion of $x$ is assumed to have a (large) constant block. We hence define
	\[
	Q_\Lambda := \{x \in \mathbb{X} \, : \, x_n = 0 \; \mbox{ for all } n \in M_\Lambda \}. 
	\]
To avoid contributions that come from the complement of $M_\Lambda$, we introduce the set
	\[
	R_\Lambda := \{ x \in \mathbb{X}\, :\, x_1 =0, \, x_{n} x_{n+1} = 10 \; \mbox{ for all } n \in \Lambda \N \setminus M_\Lambda \}
	\]
Combining both conditions, it is natural to define
	\[
	A_{\Lambda} := Q_\Lambda \cap R_\Lambda.
	\]
	
First, we will show that $A_{\Lambda} \subset \{ \underline{F} = \alpha, \overline{F} = \beta \}$ such that it suffices to bound the Hausdorff dimension of $A_{\Lambda}$ from below. 
Let $x\in A_{\Lambda}$ with alternation coding $\tau(x) = (n_i)_{i \in \N}$. Since the expansion of $x$ is constant on $[(1+\ell)\theta_k, \theta_{k+1}]$, there exists a corresponding index $i_k$ such that 
$
N_{i_k-1} \leqslant \lceil (1+\ell)\theta_k \rceil \leqslant \lfloor \theta_{k+1} \rfloor \leqslant N_{i_k}.
$
In particular,
\[
n_{i_k} \geqslant\theta_{k+1} - (1+\ell) \theta_k - 2 =  m \theta_k - 2,
\]
and by the assumption on $\theta_0$ this also implies  $n_{i_k} > \Lambda$.
On the other hand, the restriction via $R_{\Lambda}$ ensures that $n_{i_k}$ cannot be much larger. More precisely, we have $n_{i_k} \leqslant m \theta_k + 2\Lambda$ and hence
\[
n_{i_k} = m \theta_k + O(1).
\]
for every $k \in \N$. Note that for all other indices $i \in \N$ the defining condition for $R_\Lambda$ also enforces $n_i < \Lambda$. Hence, we have $I_\Lambda = \{i_k : k \in \N_0 \}$ and obtain
\begin{equation}
\label{EQ:F-theta-relation}
F_j^{\Lambda}(x) = \frac{1}{N_j^2} \sum_{i_k \leqslant j} m^2 \theta_k^2 + o(1).
\end{equation}
Clearly, this sequence attains its $\limsup$ along the subsequence with $j = i_k$ and $k \in \N$. Since $N_{i_k} \sim \theta_{k+1}$, we obtain
\[
\overline{F}(x) = \limsup_{k \to \infty} \frac{1}{\theta_{k+1}^2} \sum_{j=0}^{k} m^2 \theta_j^2
= m^2 \sum_{j=1}^\infty \frac{1}{(1+\ell+m)^{2j}}
= \frac{m^2}{(1+\ell+m)^2 - 1} = \beta,
\]
using the first identity from \eqref{EQ:l-m-alpha-relations} in the last step.
On the other hand, \eqref{EQ:F-theta-relation} implies that the $\liminf$ for $F_j(x)$ is obtained along the subsequence with $j = i_k -1$ and $k \in \N$. Since $N_{i_k - 1} \sim (1+\ell) \theta_k$, we get by a similar calculation as before
\[
\underline{F}(x) = \liminf_{k \to \infty} \frac{1}{(1+\ell)^2 \theta_k^2} \sum_{j=0}^{k-1} m^2 \theta_j^2 
= \frac{\beta}{(1+\ell)^2} = \alpha,
\]
using the definition of $\ell$ in the last step. This completes the proof for the statement that $A_{\Lambda} \subset \{ \underline{F} = \alpha, \overline{F} = \beta \}$. 

In view of Lemma~\ref{LEM:dimH-density}, one has to compute the upper density of $M_\Lambda$ in order to acquire a lower bound for the Hausdorff dimension of $A_{\Lambda}$. 
	Since the elements of $M_\Lambda$ are accumulated to the left of the positions $\theta_k$, we have that,
\begin{align*}
\overline{D}(M_\Lambda) 
& = \limsup_{k \to \infty} \frac{1}{\theta_k} \# (M_\Lambda \cap [1,\theta_k]) = \limsup_{k \to \infty} \frac{1}{\theta_k} \sum_{j=0}^{k-1} m \theta_j
= m \sum_{j=1}^{\infty} \frac{1}{(1 + \ell +m)^j}
\\ & = \frac{m}{\ell + m} = 1- f(\alpha,\beta),
\end{align*}
where we have used the second identity from  \eqref{EQ:l-m-alpha-relations} in the last step. Since the points in $A_{\Lambda}$ are determined precisely for the positions in $M_\Lambda \cup \Lambda \N \cup (\Lambda \N + 1)$, we get by Lemma~\ref{LEM:dimH-density},
\[
		\dim_H A_{\Lambda}
		\geqslant 
		1 - \overline{D} \bigl(M_\Lambda  \cup \Lambda \N \cup (\Lambda \N + 1) \bigr) 
		\geqslant
		f(\alpha,\beta) - \frac{2}{\Lambda}
		\xrightarrow{\Lambda \to \infty}
		f(\alpha,\beta).
\]
Since $\dim_H\{ \underline{F} = \alpha, \overline{F} = \beta \} \geqslant \dim_H A_{\Lambda}$, the proof is complete.
\end{proof}

\begin{coro}
\label{COR:lower-main}
Let $S \subset \Delta$. Then, $\dim_H \{ (\underline{F},\overline{F}) \in S \} \geqslant \sup_S f(\alpha,\beta)$.
\end{coro}

\begin{proof}
For $(\alpha, \beta) \in S$, we have $\{ (\underline{F},\overline{F}) \in S \} \supset \{ \underline{F} = \alpha, \overline{F}=\beta \}$ and we hence obtain $\dim_H\{ (\underline{F},\overline{F}) \in S \} \geqslant f(\alpha,\beta)$, due to Proposition~\ref{PROP:lower-bound}. Taking the supremum over $S$ yields the assertion.
\end{proof}


\section{Upper bounds}
\label{SEC:upper-bound}

We proceed by establishing an upper bound for the Hausdorff dimension of the set $\{ \underline{F} = \alpha, \overline{F} = \beta \}$. 
This is somewhat more involved than proving the lower bound because we now have to account for \emph{all} mechanisms that lead to this particular range of accumulation points.

Let us fix $x \in \mathbb{X} \setminus \mc D$ with $\tau(x) = (n_i)_{i \in \N}$ and $\Lambda \in \N$ as in the last section.
For every $k \in \N$, we define
\[
\ell_k = \ell_k(x,\Lambda)  = 
\frac{\sum_{i \in I_\Lambda \cap [1,k]} n_i}{N_k}.
\]
This corresponds to the relative density of positions of $x$ (in the region $[1,N_k]$) that are occupied with large blocks.
Naturally, if we enlarge $[1,N_k]$ by an interval that does not contain elements of $I_\Lambda$, this density decays. It will be useful to cancel this effect in an appropriate way. To that end, we define a sequence $(\varrho_k)_{k \in \N}$, implicitly dependent on $(x,\Lambda)$, via
\[
\varrho_k = \frac{\ell_k}{\sqrt{F^\Lambda_k(x)}},
\]
which we may interpret as a renormalized block density.
Indeed, one easily verifies that whenever $n_i < \Lambda$ for all $j < i 
\leqslant k$, it follows that $\varrho_j = \varrho_k$.

In the following, let
\[
\eta(\alpha,\beta) := \frac{1}{\sqrt{\beta}} (1 - f(\alpha,\beta)).
\]
In the situation of Lemma~\ref{LEM:f-beta-alpha-dof}, this may be interpreted as the relative size of the single large block $n_k$ in the region between $N_r$ and $N_k$, normalized by $\sqrt{\beta}$. 
The similarity of this interpretation with the definition of $\varrho_k$ provides some intuition for the following result.

\begin{lemma}
Whenever $n_k \geqslant \Lambda$ and $F^{\Lambda}_{k-1}(x) = \alpha < F^{\Lambda}_{k}(x) = \beta$, we can write $\varrho_k$ as the convex combination,
\[
\varrho_k = p_k \varrho_{k-1} + (1-p_k) \eta(\alpha,\beta),
\]
where
\[
p_k = \sqrt{\frac{\alpha}{\beta}} \frac{N_{k-1}}{N_k}.
\]
In particular, $p_k \leqslant N_{k-1}/N_k$.
\end{lemma}

\begin{proof}
First, we write $\ell_k$ as a convex combination via
\[
\ell_k = \frac{1}{N_k} \biggl( \sum_{i \in I_\Lambda \cap [1,k-1]} n_i + n_k \biggr) = 
\frac{N_{k-1}}{N_k} \ell_{k-1} + \frac{n_k}{N_k}. 
\]
Dividing this relation by $\sqrt{\beta}$ yields
\[
\varrho_k = \sqrt{\frac{\alpha}{\beta}} \frac{N_{k-1}}{N_k} \varrho_{k-1} + \frac{1}{\sqrt{\beta}} \frac{n_k}{N_k}.
\]
Using $n_k = N_k - N_{k-1}$, the last summand may be rewritten as
\[
\frac{1}{\sqrt{\beta}} \frac{n_k}{N_k} =
\frac{1}{\sqrt{\beta}} (1-p_k) \frac{N_k - N_{k-1}}{N_k - \sqrt{\alpha/\beta} N_{k-1}} 
= (1-p_k) \frac{1}{\sqrt{\beta}} (1 - f(\alpha,\beta)),
\]
using Lemma~\ref{LEM:f-dof-generalized} in the last step. By the definition of $\eta(\alpha,\beta)$, this is precisely the claimed expression.
\end{proof}



\begin{lemma}
\label{LEM:H-monotonicity}
The function $\eta \colon \Delta \setminus \{(0,0)\} \to [0,1]$, with
\[
\eta(\alpha,\beta) = \frac{1}{\sqrt{\beta}}(1 - f(\alpha,\beta))
\]
is continuous on its domain. It is increasing in $\alpha$ and decreasing in $\beta$.
In particular,
\[
\inf \{ \eta(\gamma, \delta) : \alpha\leqslant \gamma \leqslant \delta \leqslant \beta \} = \eta(\alpha,\beta),
\]
for all $(\alpha,\beta) \in \Delta \setminus \{(0,0)\}$.
\end{lemma}

\begin{proof}[Sketch of proof]
A short calculation shows that
\[
\eta(\alpha,\beta) = \frac{\sqrt{\beta} + \sqrt{\alpha}}{\sqrt{\alpha \beta + \beta - \alpha} + \sqrt{\alpha \beta}},
\]
which can be checked to have the required properties.
\end{proof}

\begin{prop}
\label{PROP:b_j-lower-bound}
Assume that 
$\underline{F}(x) = \alpha$ and $\overline{F}(x) = \beta$ for some $(\alpha,\beta) \in \Delta \setminus \{(0,0)\}$. 
 Then,
\[
\liminf_{k \to \infty} \varrho_k \geqslant \eta(\alpha,\beta).
\]
\end{prop}

\begin{proof}
By Lemma~\ref{LEM:F-Delta-convergence}, we have $\liminf_{m \to \infty} F_m^\Lambda(x) = \alpha$ and $\limsup_{m \to \infty} F_m^\Lambda(x) = \beta$. Given $\varepsilon > 0$ let us define $\beta_{\varepsilon} = \min\{1,\beta + \varepsilon\}$ and $\alpha_{\varepsilon} = \max\{0,\alpha- \varepsilon\}$. By assumption, we have $F_k^\Lambda(x) \in (\alpha_{\varepsilon},\beta_{\varepsilon})$ for large enough $k \in \N$. 
For such $k$ and $\gamma = F^\Lambda_{k-1}(x)$, $\delta = F^\Lambda_k(x)$, we distinguish three cases
\begin{enumerate}
\item If $n_k < \Lambda$, we have $\varrho_k = \varrho_{k-1}$.
\item If $n_k \geqslant \Lambda$ but $\gamma \geqslant \delta$, we get $\varrho_k > \varrho_{k-1}$ (by straightforward calculation).
\item If $n_k \geqslant \Lambda$ and $\gamma < \delta$, we have $\varrho_k = p_k \varrho_{k-1} + (1-p_k)\eta(\gamma,\delta)$ and $p_k \leqslant N_{k-1}/N_k$.
\end{enumerate}
Due to Lemma~\ref{LEM:H-monotonicity}, we have $\eta(\gamma,\delta) \geqslant \eta(\alpha_{\varepsilon},\beta_{\varepsilon})$ if $\gamma< \delta$. Going through all possible cases we thereby find that $\varrho_{k-1} \geqslant \eta(\alpha_{\varepsilon},\beta_{\varepsilon})$ also implies $\varrho_k \geqslant \eta(\alpha_{\varepsilon},\beta_{\varepsilon})$. By the continuity of $\eta$, the claim follows as soon as $\varrho_{k-1} \geqslant \eta(\alpha_{\varepsilon},\beta_{\varepsilon})$ for some $k$. 
Let us therefore assume that there is some $k_0 \in \N$ with $\varrho_k <\eta(\alpha_{\varepsilon},\beta_{\varepsilon})$ for all $k \geqslant k_0$.
By assumption, there are several accumulation points of the sequence $(F_k^{\Lambda}(x))_{k \in \N}$ and hence the third case needs to occur infinitely often. In each such case note that
\[
\varrho_k \geqslant p_k \varrho_{k-1} + (1- p_k)\eta(\alpha_{\varepsilon},\beta_{\varepsilon})
\]
and thereby
\[
\varrho_k -\eta(\alpha_{\varepsilon},\beta_{\varepsilon}) \geqslant p_k(\varrho_{k-1} -\eta(\alpha_{\varepsilon},\beta_{\varepsilon})).
\]
Since we have assumed that $\varrho_k$ remains below $\eta(\beta_{\varepsilon},\alpha_\varepsilon)$, this means that
\[
|\varrho_k - \eta(\alpha_{\varepsilon},\beta_{\varepsilon})| \leqslant p_k |\varrho_{k-1} - \eta(\alpha_{\varepsilon},\beta_{\varepsilon}) |.
\]
Note that $\gamma = F_{k-1}^\Lambda(x) < F_k^\Lambda(x) = \delta$ requires that $p_k \leqslant N_{k-1}/N_k$ is bounded above by some constant $c(\delta) < 1$, compare the proof of Lemma~\ref{LEM:h-mu-f-relation}. Restricting to those $k$ such that $\delta > \beta/2 > 0$, we can further assume that there is a uniform $p < 1$ with $c(\delta) < p$ and hence
\begin{equation}
\label{EQ:rho_k-to-H-decay}
|\varrho_k - \eta(\alpha_{\varepsilon},\beta_{\varepsilon})| \leqslant p |\varrho_{k-1} -\eta(\alpha_{\varepsilon},\beta_{\varepsilon}) |.
\end{equation}
Since $\varrho_k$ is non-decreasing we have overall that the distance of $\varrho_k$ to $\eta(\alpha_{\varepsilon},\beta_{\varepsilon})$ is non-increasing and exponentially decaying on a subsequence due to \eqref{EQ:rho_k-to-H-decay}. It thereby follows that $\lim_{k \to \infty} \varrho_k = \eta(\alpha_{\varepsilon},\beta_{\varepsilon})$. Hence, we have in every case
\[
\liminf_{k \to \infty} \varrho_k \geqslant \eta(\alpha_{\varepsilon},\beta_{\varepsilon}) \xrightarrow{\varepsilon \to 0} \eta(\alpha,\beta),
\]
which finishes the proof.
\end{proof}

For every $x$, let the upper density of large blocks be given by
\[
D_\Lambda(x) := \limsup_{m \to \infty} \frac{1}{N_m} \sum_{i \in [1,m] \cap I_\Lambda} n_i
= \limsup_{m \to \infty} \ell_m(x,\Lambda).
\]
From Proposition~\ref{PROP:b_j-lower-bound} we can infer the following structural property. 
\begin{prop}
\label{PROP:upper-density}
Let $S \subset \Delta$. Then, for every $x \in \bigl \{ (\underline{F},\overline{F}) \in S \bigr \}$ and for every $\Lambda \in \N$,
\[
D_\Lambda(x) \geqslant 1 - \sup\{f(\alpha,\beta): (\alpha,\beta) \in S \}.
\]
\end{prop}

\begin{proof}
Since $f(0,0) = 1$ by convention, the lower bound is trivial if $(0,0) \in S$. We can hence restrict to the case $S \subset \Delta \setminus \{(0,0)\}$. Let $x$ be such that $\underline{F}(x) = \alpha$ and $\overline{F}(x) = \beta$ with $(\alpha,\beta) \in S$. Take an increasing subsequence $(k_m)_{m \in \N}$ such that $\lim_{m \to \infty} F_{k_m}^\Lambda(x) = \beta$.
Then, by Proposition~\ref{PROP:b_j-lower-bound},
\[
D_\Lambda(x) \geqslant \liminf_{m \to \infty} \sqrt{F_{k_m}^\Lambda(x)} \varrho_{k_m}(x) \geqslant \sqrt{\beta} \, \eta(\alpha,\beta) = 1 - f(\alpha,\beta) 
\geqslant 1 - \sup_S f(\alpha,\beta).
\]
Since $\Lambda \in \N$ was arbitrary, this is the desired statement.
\end{proof}

We proceed with two results that provide an estimate for the Hausdorff dimension of level sets of the density function $D_\Lambda$. First, we recall a standard estimate, including a short proof for the reader's convenience.

\begin{lemma}
\label{LEM:variational-local-dimension}
Let $\nu$ be a probability measure on $\mathbb{X}$ and $c > 0$. Then,
\[
\dim_H \{x \in \mathbb{X} : \underline{d}_{\nu} x \leqslant c \} \leqslant c.
\]
\end{lemma}

\begin{proof}
Let $A(c) = \{ x \in \mathbb{X} : \underline{d}_{\nu}(x) \leqslant c \}$. Given $\varepsilon > 0$ and $n_0 \in \N$, choose for each $x \in A(c)$ a number $n = n(x) \geqslant n_0$ such that the cylinder $C_n(x)$ of diameter $r_n = 2^{-n}$ satisfies
\[
\nu(C_n(x)) \geqslant r_n^{c + \varepsilon}.
\] 
The open cover $F = \{C_n(x) \}_{x \in \mathbb{X}}$ has a finite subcover $G$, which can be chosen to be disjoint.
Given $s = c + \varepsilon$, we obtain
\[
\sum_{C_n(x) \in G} r_n^s \leqslant \sum_{C_n(x) \in G} \nu(C_n(x)) = 1,
\]
due to the fact that $\nu$ is a probability measure. Since $n_0 \in \N$ was arbitrary, this shows that the $s$-dimensional Hausdorff measure of $A(c)$ is finite.
Hence, $\dim_H(A(c)) \leqslant c + \varepsilon$ for all $\varepsilon > 0$ and the claim follows.
\end{proof}

\begin{lemma}
\label{LEM:D-c-upper-bound}
For $0 \leqslant c \leqslant 1$, let $B(c)$ be the set
\[
B(c) = \{ x \in \mathbb{X}\setminus \mc D : D_\Lambda(x) \geqslant 1 - c \mbox{ for all } \Lambda \in \N \}.
\]
Then, $\dim_H B(c) \leqslant c$.
\end{lemma}

\begin{proof}
We fix large integer numbers $m,k \in \N$ and set $\Lambda = km$. Let $p = 1/3$ and define a $\sigma^m$-invariant (Bernoulli) measure $\nu$ on cylinders of length $m$ via
\[
\nu([w]) = \begin{cases}
p & \mbox{ if } w \in \{0^m,1^m\},
\\p \frac{1}{2^m - 2} & \mbox{ if } w \in \{0,1 \}^m \setminus \{ 0^m,1^m\}.
\end{cases}
\]
This is extended to a product measure via the relation
\[
\nu([w_1 \cdots w_n]) = \prod_{i=1}^n \nu([w_i]),
\] 
whenever each $w_i \in \{0,1\}^m$. 
For $x \in B(c)$ with alternation coding $(n_i)_{i \in \N}$ let $j$ be such that
\begin{equation}
\label{EQ:d-Lambda-finite-bound}
 \frac{1}{N_{j}} \sum_{i \in [1,j] \cap I_\Lambda} n_i \geqslant 1-c-\varepsilon.
\end{equation}
Decompose $x^j = x_1 \cdots x_{N_j}$ into blocks of length $m$, yielding
\[
x^j = w_1 \ldots w_{r_j} \widetilde{w},
\]
where $w_i \in \{0,1 \}^m$ and $1\leqslant |\widetilde{w}| \leqslant m$. Then, due to the product definition of $\nu$,
\[
\log \nu(C_{N_j}(x)) = \sum_{r=1}^{r_j} \log \nu([w_r]) + O(1).
\]
Let $r_j^* \leqslant r_j$ be the number of indices $r$ with $w_r \in \{0^m,1^m\}$. Then,
\begin{align*}
\log \nu(C_{N_j}(x))
& = r_j^* \log(p) + (r_j-r_j^*) \log(p/(2^m - 2)) + O(1)
\\ & = r_j \log p + (r_j^{\ast} - r_j) \log(2^m - 2) + O(1).
\end{align*}
Note that for every $i \in I_\Lambda$, the number $k_i$ of words $w_r$ that are completely contained in the corresponding block of length $n_i$ satisfies
\[
k_i \geqslant \left\lfloor \frac{n_i}{m} \right\rfloor - 2.
\]
Since $n_i \geqslant \Lambda = m k$, we can choose $k$ large enough to ensure
\[
k_i \geqslant \frac{n_i}{m} (1-\varepsilon).
\] 
Hence, using \eqref{EQ:d-Lambda-finite-bound}, the number $r_j^*$ is bounded below via
\[
r_j^* \geqslant (1-\varepsilon) \frac{1}{m} \sum_{i \in [1,j] \cap I_\Lambda} n_i
\geqslant (1-\varepsilon) \frac{N_j}{m}(1 - c - \varepsilon).
\] 
As a result we get
\begin{align*}
\frac{\log \nu(C_{N_j}(x))}{N_j} 
& \geqslant \frac{r_j}{N_j} \log p - \Bigl( \frac{r_j}{N_j} - \frac{1}{m}(1-\varepsilon)(1-c-\varepsilon) \Bigr)\log(2^m - 2) + o(1)
 \\ &\xrightarrow{j \to \infty} \frac{\log p}{m} -  \frac{1}{m}\bigl(1 - (1-\varepsilon)(1-c-\varepsilon)\bigr) \log(2^m - 2).
\end{align*}
Since $\varepsilon > 0$ was arbitrary, it follows that
\[
\underline{d}_{\nu}(x) 
\leqslant \liminf_{j \to \infty} \frac{\log \nu(C_{N_j}(x))}{- N_j \log 2} 
\leqslant \frac{c \log (2^m - 2)}{m \log 2}  - \frac{\log p}{m \log 2} =: c_m.
\] 
Since this holds for all points in $B(c)$ it follows by Lemma~\ref{LEM:variational-local-dimension} that
\[
\dim_H B(c) \leqslant c_m \xrightarrow{m \to \infty} c,
\]
which indeed implies that $\dim_H B(c) \leqslant c$.
\end{proof}

\begin{coro}
\label{COR:upper-main}
For $S \subset \Delta$, we have $\dim_H \bigl\{ (\underline{F},\overline{F}) \in S \bigr\} \leqslant \sup_S f(\alpha,\beta)$.
\end{coro}

\begin{proof}
Let $c = \sup_S f(\alpha,\beta)$. Due to Proposition~\ref{PROP:upper-density}, we have $D_\Lambda(x) \geqslant 1 -c$ for all $x \in \bigl\{ (\underline{F},\overline{F}) \in S \bigr\}$ and $\Lambda \in \N$. That is, $\bigl\{ (\underline{F},\overline{F}) \in S \bigr\} \subset B(c)$ in the notation of Lemma~\ref{LEM:D-c-upper-bound}, implying that $\dim_H \bigl\{ (\underline{F},\overline{F}) \in S \bigr\} \leqslant \dim_H B(c) \leqslant c $.
\end{proof}

\begin{proof}[Proof of Theorem~\ref{THM:main}]
The lower bound in Theorem~\ref{THM:main} is given in Corollary~\ref{COR:lower-main} and the upper bound is provided by Corollary~\ref{COR:upper-main}.
\end{proof}

\section*{Achknowledgements}
The authors want to thank the Institut Mittag-Leffler  for kind hospitality during the research program ``Two Dimensional Maps" where part of this work was concluded.
PG acknowledges support from the German Research Foundation (DFG), through grant GO 3794/1-1.

\end{document}